\documentclass[3p]{elsarticle}
\usepackage{amsmath}
\usepackage{amssymb}
\usepackage{hyperref}

\newtheorem{theorem}{Theorem}
\newtheorem{claim}{Claim}
\newtheorem{corollary}{Corollary}
\newdefinition{definition}{Definition}
\newproof{example}{\textbf{Example}}
\newtheorem{lemma}{Lemma}
\newproof{notation}{{\bf Notation}}
\newtheorem{proposition}{Proposition}
\newdefinition{remark}{Remark}
\newproof{proof}{{\bf Proof}}

\numberwithin{equation}{section}

\newcommand{\I}{\infty}
\renewcommand{\S}{\sigma}

\newcommand{\T}{\tilde}
\newcommand{\N}{\nabla}

\renewcommand{\L}{\lambda}

\newcommand{\LL}{\varLambda}
\renewcommand{\r}{\varrho}

\renewcommand{\P}{\partial}
\newcommand{\B}{\beta}
\newcommand{\D}{\delta}
\newcommand{\Lap}{\triangle}
\renewcommand{\O}{\overline}
\newcommand{\A}{\alpha}
\newcommand{\lng}{\langle}
\newcommand{\rng}{\rangle}
\newcommand{\omg}{\omega}
\newcommand{\V}{\varepsilon}
\newcommand{\vp}{\varphi}
\renewcommand{\t}{\tau}
\renewcommand{\th}{\theta}

\newcommand{\QED}{\qquad\mbox{\qed}}

\newcommand{\U}{\underline}

\newcommand{\R}{\mathbb{R}}

\newcommand{\CB}{\mathcal{B}}
\newcommand{\CG}{\mathcal{G}}
\newcommand{\CE}{\mathcal{E}}
\newcommand{\CF}{\mathcal{F}}
\newcommand{\CL}{\mathcal{L}}

\newcommand{\CH}{\mathcal{H}}

\newcommand{\CK}{\mathcal{K}}
\newcommand{\CP}{\mathcal{P}}

\newcommand{\CD}{\mathcal{D}}
\newcommand{\CM}{\mathcal{M}}

\renewcommand{\CG}{\mathcal{G}}

\journal{}

\begin{document}

\begin{frontmatter}
\title{Instantaneous blow-up versus global well-posedness and asymptotic behavior for linear pseudoparabolic equations}

\author[chu]{Sujin Khomrutai\corref{cor1}}
\ead{sujin.k@chula.ac.th}

\address[chu]{Department of Mathematics and Computer Science, Faculty of Science, Chulalongkorn University, Bangkok 10330, Thailand}

\cortext[cor1]{Corresponding author.}


\begin{abstract}

We study linear pseudoparabolic equations with unbounded and time-dependent coefficients. We solve the case which has remained open in several recent studies of pseudoparabolic equations with unbounded and time-dependent coefficients. In this work we get a sharp condition for the existence or non-existence of solutions. Conditions on the initial function and coefficient are provided so that every nontrivial positive solution blows up instantaneously. When the coefficient and the initial function do not grow too rapidly, we establish the existence and uniqueness of global solutions, for both time-independent and time-dependent potentials. This is done via the analysis of the Bessel convolution multiplication operators. Asymptotic behavior and comparison principles are also established. The global well-posedness results can be extended to the equation with convection. 

\end{abstract}

\end{frontmatter}

\section{Introduction}

We study the linear pseudoparabolic Cauchy problem
\begin{align}\label{Eqn:Lin}
\begin{cases}
\P_tu-\Lap\P_tu=\Lap u+a(x,t) u&x\in\R^n,t>0,\\
\qquad u(x,0)=u_0(x)&x\in\R^n,
\end{cases}
\end{align}
and look for nontrivial solutions $u=u(x,t)\not\equiv0$. In this work, the coefficient of the lower-order term (or the \textit{potential}) is allowed to be time-dependent and unbounded. Our existence results can be extended to the more general pseudoparabolic equation with a convection term:
\begin{align}\label{Eqn:LConv}
\P_tu-\Lap\P_tu=\Lap u+\mathbf{b}(x,t)\cdot\N u+a(x,t)u,
\end{align}
where the coefficient vector $\mathbf{b}$ can be unbounded and time-dependent as well. The key idea employed to study (\ref{Eqn:LConv}) has already appeared in the investigation of (\ref{Eqn:Lin}). \smallskip

The left hand side of (\ref{Eqn:Lin}) has been arisen in various important contexts such as the {viscous diffusion equation} \cite{NovickPego91}, the {viscous Cahn-Hilliard equation} \cite{ElliottStuart96}, or pseudoparabolic regularization of the heat equation \cite{BarenBertPassoUghi93} etc.\ These are nonlinear or even higher order and degenerate equations describing various phenomena. The knowledge of the simpler looking equation (\ref{Eqn:Lin}) supplemented with those of the semilinear equations (see \cite{KhomrutaiNA15}, \cite{KhomrutaiJDE16}) should enable us to gain better understanding of such nonlinear phenomena, especially for unbounded and time-dependent coefficient equations.\smallskip

Eq.\ (\ref{Eqn:Lin}) and (\ref{Eqn:LConv}) are closely related to the linear (convective) heat equation
\begin{align*}
\P_tu=\Lap u+\mathbf{b}(x,t)\cdot\N u+a(x,t)u\quad
\end{align*}
or more generally the parabolic equation 
\begin{align*}
\P_tu=\CL_{x,t} u+a(x,t)u,\quad \CL_{x,t}=\sum a^{ij}(x,t)\P_{ij}+\mathbf{b}(x,t)\cdot\nabla.
\end{align*}
There has been a great number of studies for the linear heat and the parabolic equations with unbounded or singular and time-dependent coefficients, see for example, \cite{AronsonBesala66}, \cite{AronsonBesala67}, \cite{BrezisCabre98}, \cite{CabreMartel99}, \cite{Eidelman}, \cite{Friedman64}, \cite{KrylovPriola09}, \cite{Murata95}. Various aspects such as the existence or non-existence of solutions, uniqueness classes, fundamental solutions, comparison principles, regularity of solutions are discussed. The instantaneous blow-up phenomenon is also observed in \cite{CabreMartel99} when the equation has a singular lower-order coefficient. See also \cite{MitPohoz01,PohozaevTesei02} for applications of the celebrated test function method in the investigation of the same phenomena under various circumstances. It should be noted that the test function method however cannot be applied to the equations in this work. For linear pseudoparabolic equations, however, very few results have been known. This may be explianed from the complicated Green function. To the best of the author's knowledge, \cite{CosnerRundell83} is the only study in this direction, where some existence results are established based upon parabolic equation technique from \cite{AronsonBesala67}. For quite complete studies of the semilinear pseudoparabolic equation, see \cite{KhomrutaiNA15}, \cite{KhomrutaiJDE16}. See also \cite{AlshinKorpSvesh11} for a discussion of several important phenomena described by pseudoparabolic or more generally Sobolev equations.
\smallskip

In this work, the problem (\ref{Eqn:Lin}) will be analysed as a non-local evolution equation
\begin{align}\label{Eqn:Noloc}
\P_t\mu=\CB(V(\cdot,t)\mu)\quad(t>0),\quad\mu|_{t=0}=u_0,
\end{align}
where 
\begin{align}
\mu:=e^tu
\end{align}
and
\begin{align}
\CB:=(1-\Lap)^{-1}=\CF^{-1}\left\{(1+|\xi|^2)^{-1}\right\},\quad V:=a(x,t)+1.
\end{align}
When the problem is autonomous, i.e.\ $V$ is independent of time, then the question of existence or non-existence and the asymptotic behavior for solutions of (\ref{Eqn:Noloc}) will be reduced to the exploration of the (pointwise) limit of the series
\begin{align}\label{Exp:time-ind}
e^{t\dot\CB V}u_0:=\sum_{N=0}^\I\frac{t^N}{N!}\dot\CB V^N u_0,
\end{align}
where $\dot{\CB}V$, called the Bessel \textit{convolution multiplication operator}, and the iterations $\dot{\CB}V^N$ are defined by
\begin{align}\label{Bconv}
\begin{cases}
\displaystyle
\dot{\CB}V\vp:=\CB(V\vp)=\int_{\R^n}B(x-y)V(y,t)\vp(y)dy,\\
\vspace{-10pt}\\
\dot{\CB}V^N\vp:=\dot{\CB}V^{N-1}(\dot\CB V\vp)\quad(N\geq2).
\end{cases}
\end{align}

The main difficulty in studying (\ref{Exp:time-ind}) comes from the unboundedness of the potential, which implies that $\dot{\CB}V^N$ are unbounded operators on any weighted Lebesgue spaces $L^q(\R^n,\lng\cdot\rng^s)$ ($1\leq q\leq\I,s\in\R$). This is in contrast to the study of semilinear pseudoparabolic equation
\begin{align}\label{Eqn:Semil}
\P_tu-\Lap\P_tu=\Lap u+a(x,t)u^p,
\end{align}
where it was shown in \cite{KhomrutaiNA15,KhomrutaiJDE16} that certain powers $s$ can always be adopted so that the following key estimate for the source term holds:
\[
\|\lng\cdot\rng^s\dot{\CB}V\vp^p\|_{L^q}\lesssim\|\lng\cdot\rng^s\vp\|_{L^{q}}^p\quad\mbox{when $p\neq1$}.
\]
Then, standard contraction mapping principle techniques for parabolic equation (see for instance, \cite{Henry81}) can be applied to establish the existence of solutions for the semilinear equations within weighted Lebesgue spaces. More importantly, it was observed in \cite{KhomrutaiNA15,KhomrutaiJDE16} that, asymptotically, every nontrivial positive solution $u$ to Eq.\ (\ref{Eqn:Semil}) satisfies
\begin{align*}
u(x,t)\begin{cases}
\displaystyle\lesssim|x|^{-\A}&\mbox{for some $\A\geq\frac{\S_+}{p-1}$, 
if $p>1$},\\
\vspace{-10pt}\\
\displaystyle\gtrsim|x|^{\frac{\S_+}{1-p}}&\mbox{if $0<p<1$},
\end{cases}
\end{align*}
for all $t>0$. Thus, the solutions in the sublinear case $u\to\I$ as $p\to1^-$ whereas, in the superlinear case $u\to0$ as $p\to1^+$. Changing the weight to other weight functions, such as, the exponential functions, does not seem to solve the difficulty in the case $p=1$. In this work, the difficulty in the case $p=1$ is tackled by directly studying the pointwise convergence of the series (\ref{Exp:time-ind}). A part of this work is inspired by the general parametrix construction technique in \cite{DeckKruse02}, see also \cite{Friedman64}.
\smallskip

For non-autonomous equation, i.e.\ time-dependent potentials, we have to analyze the linear operators
\begin{align}
\begin{cases}
\displaystyle
\widehat{\CB}V(t)\vp:=\int_0^t\CB(V(\cdot,\t)\vp)\,d\t=\int_0^t\int_{\R^n}B(x-y)V(y,\t)\vp(y)\,dyd\t,\\
\vspace{-10pt}\\
\displaystyle
\widehat{\CB}V(t)^N\vp:=\widehat{\CB}V(t)^{N-1}(\widehat{\CB}V(t)\vp)\quad(N\geq2).
\end{cases}
\end{align}
The existence of solutions in this case is reduced to the convergence of the series
\begin{align}\label{Exp:time-dep}
e^{\int_0^t\dot{\CB}V(\t)\,d\t}u_0:=\sum_{N=0}^\I\widehat{\CB}V(t)^Nu_0.
\end{align}
Assuming the locally boundedness in time for the potential, one can establish the convergence using the results in the study of (\ref{Exp:time-ind}).\smallskip

In the study of instantaneous blow-up, or complete blow-up, we use the pointwise estimate involving the Bessel potential operator $\CB$ developed in \cite{KhomrutaiJDE16}. It will be revealed that the existence and uniqueness class for (\ref{Eqn:Lin}) is
\[
\CE=\left\{\vp\in L^1_{loc}(\R^n):\int_{\R^n}e^{-\r|x|}|\vp(x)|\,dx<\I,\,\,\mbox{for some $\r\in[0,1)$}\right\}.
\]
The potential also affects the blow-up behavior. We will show that if the potential has a sufficiently large spacial growth rate or a sufficiently large total change in time, then instantaneous blow-up occurs. A precise condition on the initial function and the potential will be given.
\smallskip

Now let us explain the main results of this work. We present notation and some elementary results in Section \ref{Sec:Prelim}. Then in Section \ref{Sec:Nonexist}, we investigate some non-existence results both instantaneous blowing up and non-existence of global solutions. We also give the lower bound for the assymptotic behavior. Throughout this section, we assume that the initial function satisfies
\begin{align*}
\liminf_{|x|\to\I}\frac{\log u_0}{|x|^\A}\geq\delta
\end{align*}
for some $\A,\delta\in\R$ and the potential satisfies
\begin{align}\label{Hyp:a}
a(x,t)\geq\L(t)|x|^\S\quad\mbox{for all $|x|\geq R_0,t>0$},\tag{H1}
\end{align}
where $\S\in\R$ and $R_0>0$ ($R_0\geq0$ if $\S\geq0$) are constants, and $\L=\L(t)>0$ is a locally bounded function. Conditions on $u_0,\S$, and $\L(t)$ are given so that every nontrivial nonnegative solution of (\ref{Eqn:Lin}) blows up instantaneously. In the case $\S=1$, we also prove the nonexistence of global solutions. We study the Bessel convolution multiplication operators in Section \ref{Sec:Besconv}. The convergence of (\ref{Exp:time-ind}) is proved and we introduce the Green operator. For this section, we consider arbitrary function $V=V(x,t)$ (not necessarily positive) satisfying
\begin{align}
|V(x,t)|\leq\LL(t)|x|^\S+1
\end{align}
where $\LL\in B_{loc}([0,T))$ and $0<T\leq\I$, i.e.\ $\sup_{[0,\t_0]}\LL<\I$ for all $\t_0<T$. Here, we have obtained the idea of using Lemma \ref{Lem:Ex1} and considering infinite products from \cite{DeckKruse02}. In Section \ref{Sec:ExistComp} we establish the existence and uniqueness of global solutions assuming that the potential satisfies
\begin{align}\label{Hyp:Exist}
|a(x,t)|\leq\LL(t)|x|^\S\quad\mbox{on $\R^n\times[0,\I)$}.\tag{H2}
\end{align}
We split the case of time-independent and time-dependent potential. 
Comparison principles are also proved for nonnegative potentials. An extension of the existence and uniqueness results to (\ref{Eqn:LConv}) is given in the Appendix. The equation is considered in a Volterra integral type equation and we assume that $\mathbf{b}$ is $C^1$ and both $|\mathbf{b}|$, $|\N\cdot\mathbf{b}|$ satisfy the same pointwise condition as the potential.

\section{Preliminaries}\label{Sec:Prelim}

\subsection{Notation and definitions}

\begin{notation}\
\begin{itemize}
\item[(i)] $Q_T=\R^n\times[0,T)$, $0<T\leq\I$. 
\item[(ii)] $\CB=(1-\Lap)^{-1}$ is the Bessel potential operator with kernel function $B$. 
\item[(iii)] $\displaystyle\dot\CB V\vp:=\int_{\R^n}B(x-y)V(y)\vp(y)dy$, where $V=V(x)$.
\item[(iv)] $\displaystyle\widehat{\CB}V(t)\vp:=\int_0^t\int_{\R^n}B(x-y)V(y,\t)\vp(y)dy\,d\t$, where $V=V(x,t)$. 
\item[(iv)] For $l\in\R$, we denote
\[
\begin{cases}
\displaystyle\,\liminf_{|x|\to\I}g(x)\geq l^\uparrow\,\quad\Leftrightarrow\quad\mbox{$g(x)\geq l$ on some $\{|x|\geq R\}$},\\
\displaystyle \limsup_{|x|\to\I}g(x)\leq l^\downarrow\quad\Leftrightarrow\quad\mbox{$g(x)\leq l$ on some $\{|x|\geq R\}$},
\end{cases}
\]
where $R>0$.
\item[(v)] If $u=u(x,t)$ we denote $u(t):=u(\cdot,t)$.
\item[(vi)] $\displaystyle\varPhi_\r g:=e^{-\r|\cdot|}\ast g=\int_{\R^n}e^{-\r|x-y|}g(y)dy$
\end{itemize}
\end{notation}

The solutions of (\ref{Eqn:Lin}) will be searched within the following space of functions
\begin{align}
\CE=\bigcup_{\r\in[0,1)}\CE_\r,
\end{align}
where, for each $\r\in[0,1)$,
\begin{align}
\CE_\r:=\left\{g\in L^1_{loc}(\R^n):\int_{\R^n}e^{-\r|x-y|}|g(y)|dy<\I\right\}=L^1(\R^n,e^{-\r|\cdot|}).
\end{align}

\begin{definition}

A solution of (\ref{Eqn:Lin}) is $u\in C([0,T);\CE)$, for some $0<T\leq\I$, such that $\mu:=e^tu$ satisfies
\begin{align}\label{Eqn:Milds}
\mu(t)=u_0+\int_0^t\CB(V(\cdot,\t)\mu(\t))d\t\quad\mbox{on $Q_T$}.
\end{align}
Here $\CB$ is the \textit{Bessel potential operator} where the kernel function $B$ is defined through $K_\nu$, the modified Bessel function of the second kind, by
\[
B(x)=|x|^{1-\frac{n}{2}}K_{\frac{n}{2}-1}(|x|).
\]

\end{definition}





\subsection{Some basic estimates}

\begin{lemma}\label{Lem:tech}

Let $x,y\geq0$.
\begin{enumerate}
\item[(i)] If $\A\in[0,1]$ then $(x+y)^\A\leq x^\A+y^\A$.
\item[(ii)] If $k\in[1,\I)$ then $(x+y)^k\leq2^k(x^k+y^k)$.
\end{enumerate}

\end{lemma}

\begin{proof}
(i) can be proved by elementary calculus and (ii) is true by the fact that $x+y\leq2\max\{x,y\}$.\QED
\end{proof}

We next present the two-sided estimates for $B$.

\begin{lemma}[\cite{KhomrutaiNA15},\cite{KhomrutaiJDE16},\cite{Nikolsskii75}]\label{Lem:B2side}

The Bessel potential kernel satisfies the following estimates
\begin{align}\label{Est:B2side}
\U b(x)e^{-|x|}\lesssim B(x)\lesssim\O b(x)e^{-|x|}
\end{align}
where
\begin{align*}
\U b(x):=\begin{cases}
\displaystyle|x|^{\frac{1-n}{2}}&\mbox{if $n\neq2$ or $|x|\geq1$},\\
\vspace{-10pt}\\
\displaystyle
1-\ln|x|&\mbox{if $n=2$ and $|x|<1$},
\end{cases}\quad
\O b(x)=\begin{cases}
|x|^{\frac{1-n}{2}}&|x|\geq1,\\
\vspace{-10pt}\\
1-\ln|x|&|x|<1\,\,\mbox{and $n=2$},\\
\vspace{-10pt}\\
|x|^{2-n}&|x|<1\,\,\mbox{and $n\geq3$},\\
\vspace{-10pt}\\
1&|x|<1\,\,\mbox{and $n=1$}. 
\end{cases}
\end{align*}

\end{lemma}

In the study of nonexistence results, the following two estimates involving the Bessel potential operator will be used often.

\begin{lemma}\label{Lem:Non1}

Let $\D\in\R$, $\A\in[0,1]$, and $\gamma\in(0,1)$. There is $d_0=d_0(n,\D,\A,\gamma)>0$ such that if $d\geq d_0$ then
\begin{align}
\CB\left(|\cdot|^de^{\D|\cdot|^\A}\right)\geq \gamma|x|^de^{\D|x|^\A}.
\end{align}

\end{lemma}

\begin{proof}
We denote $F=r^de^{\D r^\A}$, where $r=|x|$, and $\kappa:=\gamma^{-1}-1>0$. We show that $\kappa F+\Lap F\geq0$ provided $d$ is large enough. Consider
\begin{align}
\Lap F&=e^{\D r^\A}\Lap r^d+2(r^d)'(e^{\D r^\A})'+r^d\Lap e^{\D r^\A},\nonumber\\
&=\left[(r^d)''+\frac{n-1}{r}(r^d)'\right]e^{\D r^\A}+2d\D\A r^{d+\A-2}e^{\D r^\A}+\left[(e^{\D r^\A})''+\frac{n-1}{r}(e^{\D r^\A})'\right]r^d,\nonumber\\
&=\left(\D^2\A^2r^{2\A}+\D\A(2d+\A+n-2)r^\A+d(d+n-2)\right)r^{d-2}e^{\D r^\A}.\label{Lem1Est:1}
\end{align}
If $\D\geq0$ or $\A=0$ then 
\[
\kappa F+\Lap F\geq0\quad\mbox{provided $d+n-2\geq0$}.
\]
For $r<1$, $\kappa F+\Lap F\geq0$ by taking $d$ sufficiently large.\smallskip

Assume $\D<0$, $\A>0$, and $r\geq1$. Since $\A\leq1$, we get from (\ref{Lem1Est:1}) that
\begin{align*}
\kappa F+\Lap F\geq Q(r^\A)r^{d-2}e^{\D r^\A},
\end{align*}
where 
\[
Q(X)=(\kappa+\D^2\A^2)X^2+\D\A(2d+\A+n-2)X+d(d+n-2).
\]
The discriminant of $Q$ is
\begin{align*}
\varDelta=4d^2\D^2\A^2\left[\left(1+\frac{\A+n-2}{2d}\right)^2-\left(1+\frac{\kappa}{\D^2\A^2}\right)\left(1+\frac{n-2}{d}\right)\right].
\end{align*}
Since $\varDelta<0$ as $d\to\I$, we obtain by taking $d$ sufficiently large that
\[
\kappa F+\Lap F\geq0\quad(\mbox{on $\R^n$}).
\]
This implies
\[
\gamma^{-1}F\geq F-\Lap F,
\]
hence by taking $\CB$ we conclude that $\CB F\geq \gamma F$. This implies what we want.\QED
\end{proof}

We also need the following lemma where the proof is almost the same as \cite{KhomrutaiJDE16}.

\begin{lemma}\label{Lem:BlowerExp}

Let $d\geq0$ and $\A\in[0,1]$. Assume $\D\in\R$ if $\A<1$ and $\D\in(-\I,1)$ if $\A=1$. There is a constant $\eta_0=\eta_0(n,\D,\A,R)>0$ such that
\begin{align}\label{Est:BLalpha}
&\CB\left(|\cdot|^de^{\D|\cdot|^\A}\cdot1_{\R^n\setminus B_{R}}\right)\geq \eta_02^{-d}|x|^{d}e^{\D|x|^\A}.
\end{align}

\end{lemma}

\begin{proof}

By Lemma \ref{Lem:tech} (i) we have $|x|^\A-|y|^\A\leq|x-y|^\A\leq|x|^\A+|y|^\A$ which implies $\D|x-y|^\A\leq\D|x|^\A+|\D||y|^\A$. Then by applying the lower estimate in Lemma \ref{Lem:B2side} we get
\begin{align*}
\CB\left(|\cdot|^de^{\D|\cdot|^\A}\cdot1_{\R^n\setminus B_{R}}\right)
&=\int_{|x-y|\geq R}B(y)|x-y|^{d}e^{\D|x-y|^\A}dy,\\
&\geq\int_{|x-y|\geq R}\U b(y)e^{-|y|}|x-y|^{d}e^{\D|x|^\A-|\D||y|^\A}dy,\\
&=e^{\D|x|^\A}\int_{|x-y|\geq R}\U b(y)e^{-(|y|+|\D||y|^\A)}|x-y|^{d}dy=:e^{\D|x|^\A}K.
\end{align*}
We can assume $R\geq1$. If $|x|\leq 2R$ then for $|y|\geq4R$ we get $|x-y|\geq\max\{2R,|x|\}$, hence
\begin{align*}
K&\geq\omg_n|x|^{d}\int_{4R}^{\I}\U b(r)e^{-(r+|\D|r^\A)}r^{n-1}dr=:\eta_1|x|^{d}.
\end{align*}
On the other hand, if $|x|\geq2R$ then for $|y|\leq R$ we get $|x-y|\geq|x|/2\geq R$, hence
\begin{align*}
K&\geq\omg_n2^{-d}|x|^{d}\int_0^{R}\U b(r)e^{-(r+|\D|r^\A)}r^{n-1}dr=:\eta_22^{-d}|x|^{d}.
\end{align*}
Let $\eta_0=\eta_1\wedge \eta_2>0$. Then we obtain
\[
\CB\left(|\cdot|^de^{\D|\cdot|^\A}\cdot1_{\R^n\setminus B_{R}}\right)\geq\eta_02^{-d}|x|^{d}e^{\D|x|^\A},
\]
which is the desired estimate. 
\QED
\end{proof}




\section{Nonexistence of positive solutions}\label{Sec:Nonexist}

In this section we assume the solutions $u$ and initial condition $u_0$ of (\ref{Eqn:Lin}) are nonnegative functions. We start with the following fact for the nontrivial nonnegative solutions. It is true owing to the nonlocal nature of the equation and that the Bessel potential kernel is decaying like $e^{-|x|}$, see (\ref{Est:B2side}).

\begin{lemma}[\cite{KhomrutaiJDE16}, Lemma 8]\label{Lem:Non2}

Let $0\lneq u\in C(Q_T)$ be a solution of (\ref{Eqn:Lin}) in $Q_T$. Assume $\t_0\in(0,T)$ and $\D>1$. Then
\begin{align}\label{Decay:explin}
u(\t_0)\gtrsim e^{-\D|x|}\quad\mbox{as $|x|\to\I$},\quad\mbox{or, equivalently},\quad\liminf_{|x|\to\I}\frac{\log u(\t_0)}{|x|}\geq-\D^\uparrow.
\end{align}

\end{lemma}

\begin{remark} By (\ref{Decay:explin}) and the semigroup property for (\ref{Eqn:Lin}), we will assume without loss of generality that
\begin{align}\label{Hyp:u0pos}
u_0>0\quad\mbox{on $\R^n$}.\tag{H3}
\end{align}
Also observe that (\ref{Decay:explin}) implies that $u>0$ on $\R^n\times(0,\I)$.
\end{remark}

We note that $u\gtrsim\exp(l|x|^m)$ as $|x|\to\I$ is equivalent to that
\[
\liminf_{|x|\to\I}\frac{\log u}{|x|^m}\geq l.
\]

\begin{remark}\label{Rem:u0decay}\
\begin{enumerate}
\item[(i)] The estimate (\ref{Decay:explin}) implies that the fastest possible decay rate for any nontrivial solution of (\ref{Eqn:Lin}) at any time $t>0$ is the \textit{exp-linear} function
\[
e^{-(1+\V)|x|}\quad(\V>0),
\]
even if $u_0$ has a \textit{rapid (exp-superlinear) decay rate}: 
\[
\exp\left(\D|x|^{\A}\right)\quad(\D<0,\A>1).
\]
If $u_0$ has a rapid decay rate, we will consider instead a time $\t_0$-translated problem where the shifted initial condition $u(\t_0)$ satisfies (\ref{Decay:explin}). On the other hand, if $\D>0$ and $\A>1$, i.e.\ $u_0$ has a \textit{rapid (exp-superlinear) growth rate}, it will be shown in Proposition \ref{Prop:fastgrow} that solutions blow up completely regardless of the potential.
\item[(ii)] If $u_0$ has a \textit{slow (exp-sublinear) decay} ($\D<0$) or \textit{growth} ($\D\geq0$) \textit{rate}:
\[
\exp\left(\D|x|^{\A}\right)\quad(0\leq\A<1),
\]
then it will be shown that a solution exhibits, at any time $t>0$, at least the same rate. Finally, we note that, asymptotically
\[
\exp(\D|x|^{\A})\sim\mbox{a constant}>0\quad\mbox{if $\A<0$}.
\]
\end{enumerate}

\end{remark}

\begin{proposition}[Instantanneous blow-up I]\label{Prop:fastgrow}

Let $\CK$ be a cone-like domain or a tube domain with one end at infinity and $\D>0$, $\A\geq1$ be constants such that
\begin{align*}
\begin{cases}
\D\in(0,\I)&\mbox{if $\A>1$,}\\
\D\in[1,\I)&\mbox{if $\A=1$}.
\end{cases}
\end{align*}
If $u_0$ satisfies
\begin{align}
\liminf_{|x|\to\I,x\in\CK}\frac{\log u_0}{|x|^\A}\geq\D^\uparrow,
\end{align}
then (\ref{Eqn:Lin}) with arbitrary potential $a(x,t)\geq0$ has no solutions on any $Q_T$; in fact, every solution blows up completely.

\end{proposition}

\begin{proof}
We note that $u\geq\CG(t)u_0:=\sum_{k=0}^\I\frac{1}{k!}\CB^ku_0$, where $\CG(t)$ denotes the pseudoparabolic green operator, which implies in particular that $u\geq\CB u_0$. Since $u_0\geq Ce^{|x|}$ as $|x|\to\I$, $x\in\CK$, we get
\[
u(x,t)\geq C\int_{|x-y|\geq1,y\in\CK,|y|\gg1}|x-y|^{\frac{1-n}{2}}e^{-|x-y|}e^{|y|}dy=\I,
\]
which is true for all $x\in\R^n$, $t>0$.\QED
\end{proof}

According Remark \ref{Rem:u0decay} (i) and Proposition \ref{Prop:fastgrow}, it remains to consider the behavior of solutions to (\ref{Eqn:Lin}) assuming that the initial function satisfies
\begin{align}\label{Hyp:u0alpha}
\begin{cases}
\displaystyle\liminf_{|x|\to\I}\frac{\log u_0}{|x|^\A}\geq\D^\uparrow&(\D\in\R,\A\leq1),\mbox{where}\\
\vspace{-10pt}\\
\D\in\R&\hspace{-20pt}\mbox{if $\A<1$ (slow decay/growth)},\\
\D\in(-\I,1)&\hspace{-20pt}\mbox{if $\A=1$ (exp-linear decay/growth)}.
\end{cases}
\end{align}
We prove the following a priori lower bound for solutions affected by the potential and the initial condition.

\begin{theorem}[Asymptotic behavior]
\label{Thm:Lower}

Assume (\ref{Hyp:a}), (\ref{Hyp:u0pos}), and (\ref{Hyp:u0alpha}). Let $0\lneq u\in C(Q_T)$ be a solution of (\ref{Eqn:Lin}) in $Q_T$ where $T<\I$. Then, for any $\t_0>0$ and $0<\V<1$, the following estimate is true
\begin{align}
\mu\gtrsim\exp\left((1-\V)c_{T}\L_\ast(t)|x|^\S+\D|x|^{\A_+}\right)\quad\mbox{on $\R^n\times[\t_0,T)$},
\end{align}
where
\begin{align}
c_{T}:=\min\left\{1,\frac{R_0^{-\S}}{\|\L\|_{L^\I(0,T)}}\right\}.
\end{align}

\end{theorem}

\begin{proof}
In view of Remark \ref{Rem:u0decay} (ii), if $\A<0$ we get $u_0\geq\mbox{const}>0$. Then it can be seen from the proof below that $u_0$ does not affect all the estimates. We will assume that
\[
0\leq\A\leq1.
\]
Observe that the assumption (\ref{Hyp:u0alpha}) is the same as $u_0\geq C_0e^{\D|x|^\A}$.\smallskip

By (\ref{Hyp:a}), we have 
\begin{align}\label{Est:VLow}
V\geq1+\L(t)|x|^\S\cdot1_{\R^n\setminus B_{R_0}}.
\end{align}
Let $\eta_0>0$ be the constant in Lemma \ref{Lem:BlowerExp}. We prove the following claim.

\begin{claim}\label{Claim:Explower1}

For all positive integer $N$, we have 
\begin{align}
\mu\geq C_0\eta_0^N2^{-\frac{N(N+1)}{2}\S}\frac{1}{N!}\L_\ast(t)^N|x|^{\S N}e^{\D|x|^\A}
\end{align}
\end{claim}

\begin{proof}[\textbf{Claim}]
We employ Lemma \ref{Lem:BlowerExp} and (\ref{Est:VLow}). By (\ref{Eqn:Milds}) and that $u\geq u_0\geq C_0e^{\D|x|^\A}=:I_0$, we have
\begin{align*}
&\mu\geq \int_0^t(\dot\CB VI_0)(\t)d\t\geq\int_0^tC_0\eta_0\left(1+\L(\t)2^{-\S}|x|^\S\right)e^{\D|x|^\A}d\t,\\
&\hphantom{\mu}\geq C_0\eta_02^{-\S}\left(t+\L_\ast(t)|x|^\S\right)e^{\D|x|^\A}=:I_1,\\
&\mu\geq\int_0^t(\dot\CB VI_1)(\t)d\t\geq \int_0^tC_0\eta_0^22^{-\S}\left(1+\L(\t)2^{-\S}|x|^\S\right)\left(\t+\L_\ast(\t)2^{-\S}|x|^\S\right)e^{\D|x|^\A}d\t,\\
&\hphantom{\mu}\geq C_0\eta_0^22^{-3\S}\frac{1}{2!}\left(t+\L_\ast(t)|x|^\S\right)^2e^{\D|x|^\A}=:I_2,
\end{align*}
where we have applied the fact that
\[
1+\L(\t)2^{-\S}|x|^\S=\frac{\P}{\P\t}(\t+\L_\ast(\t)2^{-\S}|x|^\S).
\]
Similarly, we get
\begin{align*}
&\mu\geq\int_0^t(\dot\CB V I_2)(\t)d\t\geq \int_0^tC_0\eta_0^32^{-3\S}\frac{1}{2!}\left(1+\L(\t)2^{-\S}|x|^\S\right)\left(\t+\L_\ast(\t)2^{-\S}|x|^\S\right)^2e^{\D|x|^\A}d\t,\\
&\hphantom{\mu}\geq C_0\eta_0^32^{-6\S}\frac{1}{3!}\left(t+\L_\ast(t)|x|^\S\right)^3e^{\D|x|^\A}=:I_3.
\end{align*}
By induction, it follows that
\begin{align*}
\mu\geq C_0\eta_0^N2^{-\frac{N(N+1)}{2}\S}\frac{1}{N!}(t+\L_\ast(t)|x|^\S)^Ne^{\D|x|^\A}\qquad\forall\,N>0.
\end{align*}
This implies the desired estimate of the claim.\QED
\end{proof}

Let $\V\in(0,1)$ and $d_0=d_0(n,\D,\A,1-\V)$ as in Lemma \ref{Lem:Non1}. In the preceding claim, let us choose $N=N(\V)$ sufficiently large so that
\[
d:=\S N\geq d_0.
\]
Also fix $\t_0\in(0,T)$. It follows from the claim that
\[
\mu\geq C_1|x|^de^{\D|x|^\A}\quad\mbox{on $\R^n\times[\t_0,T)$},
\]
where $C_1:=C_0\eta_0^N2^{-(N(N+1)/2)\S}(1/N!)\L_\ast(\t_0)^N$. \smallskip

Let $v:=u(t+\t_0)$. So $v$ satisfies 
\[
\P_tv-\Lap\P_tv=\Lap v+a(x,t+\t_0)v,\quad v_0=u(\t_0).
\]
$a(\cdot,t+\t_0)$ satisfies the same hypotheses as $a(\cdot,t)$ with an obvious modification and 
\begin{align}
v_0\geq J_0:=C_2|x|^de^{\D|x|^\A}\quad\mbox{on $\R^n$},
\end{align}
where $C_2:=C_1e^{-\t_0}$. The lower bound (\ref{Est:VLow}) implies that $W:=1+a(x,t+\t_0)$ satisfies
\begin{align}
\begin{cases}
\displaystyle W\geq\LL(t)|x|^\S\quad(x\in\R^n,t>0),\quad\mbox{where}\\
\vspace{-10pt}\\
\displaystyle\LL(t)=c_T\L(t+\t_0)>0,\,\, c_T:=\min\left\{1,\frac{R_0^{-\S}}{\|\L\|_{L^\I(0,T)}}\right\}.
\end{cases}
\end{align}
Since $T<\I$, we have $c_T>0$ and $\LL_\ast(t)\in(0,\I)$ for all $t$.\smallskip

Next we repeatedly apply Lemma \ref{Lem:Non1} with $\gamma=1-\V$ to get the following result.

\begin{claim}\label{Claim:ExpLower2}
If $\nu=e^tv$, then we have
\begin{align}
\nu\geq C_2|x|^d\exp\left((1-\V)\LL_\ast(t)|x|^\S+\D|x|^\A\right),\quad\LL_\ast(t):=\int_0^t\LL(\t)d\t.
\end{align}
\end{claim}

\begin{proof}[\textbf{Claim}]
Since $\nu$ satisfies $\nu=v_0+\int_0^t(\dot\CB W\nu)(\t)d\t$ and $\nu\geq v_0\geq J_0$, we have by Lemma \ref{Lem:Non1} that
\begin{align*}
&\nu\geq J_0+\int_0^t(\dot\CB W J_0)(\t)d\t,\\
&\hphantom{\nu}\geq J_0+C_2(1-\V)\int_0^t\LL(\t)|x|^{\S+d}e^{\D|x|^\A}d\t,\\
&\hphantom{\nu}=J_0+C_2(1-\V)\LL_\ast(t)|x|^{\S+d}e^{\D|x|^\A}=:J_0+J_1,
\end{align*}
and
\begin{align*}
&\nu\geq J_0+\int_0^t(\dot\CB W(J_0+J_1))(\t)d\t\geq J_0+J_1+\int_0^t(\dot\CB W J_1)(\t)d\t,\\
&\hphantom{\nu}\geq J_0+J_1+C_2(1-\V)^2\int_0^t\LL(\t)\LL_\ast(\t)|x|^{2\S+d}e^{\D|x|^\A}d\t,\\
&\hphantom{\nu}\geq J_0+J_1+C_1\frac{1}{2!}((1-\V)\LL_\ast(t))^2|x|^{2\S+d}e^{\D|x|^\A}=:J_0+J_1+J_2.
\end{align*}
Similarly, we have
\begin{align*}
&\nu\geq J_0+\int_0^t(\dot\CB W(J_0+J_1+J_2))(\t)d\t\geq J_0+J_1+J_2+\int_0^t(\dot\CB WJ_2)(\t)d\t,\\
&\hphantom{\nu}\geq\sum_{k=0}^2J_k+C_2(1-\V)^3\int_0^t\LL(\t)\frac{1}{2!}\LL_\ast(\t)^2|x|^{3\S+d}e^{\D|x|^\A}d\t,\\
&\hphantom{\nu}=\sum_{k=0}^2J_k+C_2\frac{1}{3!}((1-\V)\LL_\ast(t))^3|x|^{3\S+d}e^{\D|x|^\A}=:\sum_{k=0}^3J_k.
\end{align*}
By induction we obtain for any positive integer $N$ that
\begin{align*}
\nu\geq\sum_{k=0}^NC_2\frac{1}{k!}((1-\V)\LL_\ast(t))^k|x|^{k\S+d}e^{\D|x|^\A}.
\end{align*}
This is true for all $N$, hence we obtain
\begin{align*}
\nu\geq C_2|x|^de^{\D|x|^\A}\sum_{k=0}^\I\frac{((1-\V)\LL_\ast(t))^k}{k!}|x|^{k\S}=C_2|x|^de^{(1-\V)\LL_\ast(t)|x|^\S+\D|x|^\A},
\end{align*}
which is the desired estimate.\QED
\end{proof}

Since $\mu(t+\t_0)=e^{\t_0}\nu$ it follows from the preceding claim that 
\begin{align*}
\mu(t+\t_0)&\geq C_1|x|^d\exp\left((1-\V)c_{T}\L_\ast(t)|x|^\S+\D|x|^\A\right).
\end{align*}
On $S=\left\{|x|\leq1,t\in[\t_0,T]\right\}$, we set 
\begin{align*}
m:=\min_S\mu(x,t),\quad
l:=\max_S\exp\left((1-\V)c_{T}\L_\ast(t)|x|^\S+\D|x|^d\right)
\end{align*}
which are positive real numbers. Then we have
\begin{align*}
\mu&\geq\begin{cases}
\displaystyle m&\mbox{if $(x,t)\in S$},\\
\vspace{-10pt}\\
\displaystyle C_1\exp\left((1-\V)c_{T}\L_\ast(t)|x|^\S+\D|x|^\A\right)&\mbox{if $|x|\geq1,t\in[\t_0,T]$}.
\end{cases}
\end{align*}
which implies that
\[
\mu(x,t)\geq C_3\exp\left((1-\V)c_{T}\L_\ast(t)|x|^\S+\D|x|^\A\right)\quad(x\in\R^n,t>\t_0),
\]
where $C_3:=C_1\min\{1,mC_1^{-1}l^{-1}\}$.\QED

\end{proof}

Our next aim is to prove some non-existence results for the Cauchy problem (\ref{Eqn:Lin}). 

\begin{theorem}
Assume (\ref{Hyp:a}) and (\ref{Hyp:u0pos}). Assume further that (\ref{Hyp:u0alpha}) is true.
\begin{enumerate}
\item[(i)] (\textbf{Instantaneous blow-up II}) If $\S>1$, then Eq.\ (\ref{Eqn:Lin}) has no nontrivial positive solution on any $Q_T$. In fact, every such solution blows up completely on $Q_T$.
\item[(ii)] (\textbf{Nonexistence of global solutions}) If $\S=1$ with $\L>0$ satisfies
\begin{align}\label{Hyp:LLnon}
\sup_{\t>0}c_{\t}\L_\ast(\t)>\begin{cases}
1&\mbox{if $\A<1$},\\
1-\D_-&\mbox{if $\A=1$},
\end{cases}
\end{align}
then Eq.\ (\ref{Eqn:Lin}) has no nontrivial global solution. 
\end{enumerate}

\end{theorem}

\begin{proof}
(i) 
Assume $u\gneq0$ is a solution of (\ref{Eqn:Lin}) on some $Q_T$. Fix any $\t_0>0$. By Theorem \ref{Thm:Lower} then
\[
\mu\geq C\exp\left((1-\V)c_{T}\L_\ast(t)|x|^\S+\D|x|^{\A_+}\right)\quad(x\in\R^n,\t_0\leq t<T).
\]
We note that $\S>1\geq\A_+$. Define
\[
L(t):=\left(\frac{2|\D|}{(1-\V)c_{T}\L_\ast(t)}\right)^{1/(\S-\A_+)}\quad(t\geq\t_0).
\]
If $|x|>L(t)$, then $(1-\V)c_{T}\L_\ast(t)|x|^\S+\D|x|^{\A_+}\geq(1/2)(1-\V)c_{T}\L_\ast(t)|x|^\S$, hence
\[
\mu\geq C\exp\left(\frac{1}{2}(1-\V)c_{T}\L_\ast(t)|x|^\S\right)\quad\mbox{on $\{|x|>L(t),t\geq\t_0\}$}.
\]
Observe that $t\mapsto L(t)$ is a decreasing function and $L(t)\to\I$ as $\V\to1$.\smallskip

Fix $\t_1\in(\t_0,T)$ and define
\begin{align}
l:=\inf_{|x|\leq L(\t_0),t\in[\t_0,\t_1]}\mu(x,t)>0.
\end{align}
For each $t\in[\t_0,\t_1]$, we have
\begin{align*}
\displaystyle\mu(x,t)\geq
\begin{cases}
l&\mbox{if $|x|\leq L(\t_0)$},\\
\vspace{-10pt}\\
\displaystyle C\exp\left(\frac{1}{2}(1-\V)c_{T}\L_\ast(t)|x|^\S\right)&\mbox{if $|x|>L(\t_0)$}.
\end{cases}
\end{align*}
This implies
\begin{align}
\mu(x,t)\geq C'e^{\B|x|^\S}\quad(x\in\R^n,t\in[\t_0,\t_1]),
\end{align}
where 
\begin{align}
\B:=\frac{1}{2}(1-\V)c_{T}\L_\ast(\t_0),\quad
C':=e^{-\B L(\t_0)^\S}\min\left\{l,Ce^{\B L(\t_0)^\S}\right\}.
\end{align}

Now for any $x\in\R^n$ and $t\in[\t_0,\t_1]$, we have
\begin{align*}
\mu(x,t)&=u_0+\int_0^t(\dot\CB V\mu)(x,\t)d\t\geq\int_{0}^t(\CB\mu)(x,\t)d\t,\\
&\geq\int_{\t_0}^t\int_{\R^n}\U b(y)e^{-|y|}C'e^{\B|x-y|^\S}dy,\\
&=(t-\t_0)C'\int_{\R^n}\U b(y)e^{\B|x-y|^\S-|y|}dy=:(t-\t_0)C'K(x).
\end{align*}
We consider the integral $K(x)$. If $|y|\geq1$ then $\U b(y)\geq C|y|^{(1-n)/2}$. Also if in addition $|y|\gg|x|$, says
\[
|y|\geq r_0:=\max\left\{\frac{1}{1-\nu}|x|,\left(\frac{1+\V_0}{\B\nu^\S}\right)^{\frac{1}{\S}}\right\},\quad(0<\nu<1,\V_0>0),
\]
then we have 
\begin{align*}
\B|x-y|^\S-|y|&\geq\B(|y|-|x|)^\S-|y|\geq\B\nu^\S|y|^\S-|y|,\\
&\geq|y|\B\nu^\S\cdot\frac{1+\V_0}{\B\nu^\S}-|y|\geq\V_0|y|.
\end{align*}
Hence
\begin{align*}
K&\geq\T C\int_{|y|\geq1\wedge r_0}|y|^{\frac{1-n}{2}}e^{\V_0|y|}dy=\I.
\end{align*}
So $u$ instantaneously blows up on $\R^n\times[\t_0,T)$ for any $\t_0>0$, implying the assertion (i).\smallskip

(ii) Assume that $u\geq0$ is a nontrivial global solution for Eq.\ (\ref{Eqn:Lin}). By the assumption (\ref{Hyp:LLnon}), we can choose $\t_0,T$ sufficiently large with $\t_0<T$ and $\V>0$ close to 0 so that
\begin{align*}
(1-\V)c_{T}\L_\ast(t)>\begin{cases}
1&\mbox{if $\A<1$},\\
1-\D_-&\mbox{if $\A=1$}.
\end{cases}\quad\mbox{for all $t\in[\t_0,T]$}.
\end{align*}
By Theorem \ref{Thm:Lower} we then have
\begin{align*}
\mu\gtrsim e^{|x|}
\end{align*}
for all $|x|$ sufficiently large, uniformyly for $t\in[\t_0,T]$. Using that $\mu\geq\int_0^t\CB\mu(\t)\,d\t$
it then follows that
\begin{align*}
\mu(x,t)&\geq C(t-\t_0)\int_{\R^n}\U b(x-y)e^{-|x-y|}e^{|y|}dy,\\
&\geq C(t-\t_0)\int_{|x-y|\geq1,|y|\gg1}|x-y|^{\frac{1-n}{2}}e^{-|x|}dy=\I,
\end{align*}
for all $x\in\R^n,\t_0<t\leq T$.

\end{proof}

\begin{corollary}
Let $\S\geq1$. Then the equation
\begin{align}
\P_tu-\Lap\P_tu=\Lap u+t^\nu a(x)u\quad(x\in\R^n,t>0),
\end{align}
where $\nu\geq0$ $a(x)\gtrsim|x|^\S$ as $|x|\to\I$, has no solutions $u\gneq0$ in the case $\S>1$ and it has no nontrivial global solutions $u\gneq0$ when $\S=1$. In particular, the problem
\begin{align}
\P_tu-\Lap\P_tu=\Lap u+t^\nu|x|^\S u\quad(x\in\R^n,t>0)
\end{align}
has no solutions $u\gneq0$ if $\S>1$ and it has no global solutions $u\gneq0$ if $\S=1$.

\end{corollary}

\begin{corollary}

Let $\L$ be a constant. Assume that
\begin{align}
\L>1\quad\mbox{and}\quad\liminf_{|x|\to\I}\log u_0\geq0.
\end{align}
Then the Cauchy problem
\begin{align}
\P_tu-\Lap\P_tu=\Lap u+\L e^{-t}|x|u,\quad u(0)=u_0
\end{align}
has no global solution $u\gneq0$.

\end{corollary}







\section{Bessel convolution multiplication operator}\label{Sec:Besconv}

In this section we investigate the Bessel convolution multiplication operator and its iterations (\ref{Bconv}) which arise in the study of (\ref{Eqn:Lin}) and (\ref{Eqn:LConv}). Assume $V:\R^n\times[0,T)\to\R$ ($0<T\leq\I$) is a real-valued function that has at most a power spacial growth at infinity; precisely, there is $\S\geq0$ such that
\begin{align}\label{Hyp:Vtime}
|V(x,t)|\leq\LL(t)|x|^\S+1\quad\mbox{where $\LL\in B_{loc}([0,T))$}.
\end{align}
Our main goal is to study the one-parameter semigroup
\begin{align}
e^{t\dot\CB V}=\sum_{k=0}^\I\frac{t^k}{k!}\dot\CB V^k,
\end{align}
which arises from the evolution equation (\ref{Eqn:Lin}).\smallskip

All the results can be generalized to more general operators, especially, in the study of (\ref{Eqn:LConv}). In fact, we can consider any convolution operator $\CH$:
\begin{align}
\CH\vp=\int_{\R^n}H(x-y)\vp(y)\,dy,
\end{align}
such that the kernel function $H$ satisfies
\begin{align}\label{K}
H(x)=H(|x|)\geq0,\quad H\in L^1_{loc}(\R^n),\quad H(x)\lesssim e^{-|x|}\quad\mbox{as $|x|\to\I$}.
\end{align}

\begin{lemma}\label{Lem:Ex1}

Let $D\in[0,\I)$, $\gamma\in(0,\I)$, and $\V\in(0,1)$. Then
\begin{align}
|y|^De^{-\gamma|x-y|}\leq\left(|x|+\frac{D}{(1-\V)\gamma}\right)^De^{-\V\gamma|x-y|}\quad(\forall\,x,y\in\R^n).
\end{align}

\end{lemma}

\begin{proof}
By homogenization, it suffices to show that
\[
|y|e^{-|x-y|}\leq|x|+1\quad\forall\,x,y\in\R^n.
\]
But this is true because $(|y|-|x|)e^{-|x-y|}\leq|x-y|e^{-|x-y|}\leq1\leq|x|(1-e^{-|x-y|})+1$. \QED
\end{proof}

Due to the fact that the kernel $B\in L^1_{loc}(\R^n)$, $B\lesssim e^{-|x|}$ at infinity, and $V$ has at most a power growth, we have the following estimates for $\dot{\CB}V$.

\begin{lemma}\label{Lem:Ex2}
There is $c_0=c_0(n)>0$ such that, for any $\LL_0>0$, $D\geq0$, and $\V\in(0,1)$, we have
\begin{align}
\left|\CB\left((\LL_0|\cdot|^D+1)g\right)\right|\leq c_0\left[\LL_0\left(|x|+\frac{D}{1-\V}\right)^D+1\right]\varPhi_\V|g|,
\end{align}
where $g:\R^n\to\R$ and
\begin{align}
\varPhi_\V|g|=e^{-\V|\cdot|}\ast|g|.
\end{align}
\end{lemma}

\begin{proof}
We can assume $g\geq0$. Using the upper estimate (\ref{Est:B2side}) for the kernel $B$ then we have
\begin{align*}
\CB\left(|\cdot|^Dg\right)&=\int_{\R^n}B(x-y)|y|^Dg(y)dy,\\
&\lesssim\int_{\R^n}\O b(x-y)e^{-|x-y|}|y|^Dg(y)dy,\\
&\lesssim\int_{|x-y|<1}\O b(x-y)e^{-|x-y|}|y|^Dg(y)dy+\int_{|x-y|\geq1}e^{-|x-y|}|y|^Dg(y)dy.
\end{align*}
By Lemma \ref{Lem:Ex1}, with $\gamma=1$, we have
\[
|y|^De^{-|x-y|}\leq\left(|x|+\frac{D}{1-\V}\right)^De^{-\V|x-y|}.
\]
Using Young's inequality, it follows that $\int_{|x-y|<1}\O b(x-y)g(y)dy\lesssim\int_{|x-y|<1}g(y)dy$, hence
\begin{align*}
\CB\left(|\cdot|^Dg\right)&\lesssim\left(|x|+\frac{D}{1-\V}\right)^\D\int_{\R^n}e^{-\V|x-y|}g(y)dy.
\end{align*}
Summing this estimate together with the case $D=0$, the desired estimate then follows.\QED

\end{proof}

\begin{remark}

The estimate in this lemma is true for any operator $\CH$ satisfying (\ref{K}).

\end{remark}

Let us study a typical case that $V=V_0$ where
\begin{align}\label{Hyp:V0}
V_0=\LL_0|x|^\S+1,
\end{align}
for some constants $\LL_0>0$ and $\S\geq0$. For convenience, in the proof of the following result we will denote the power function
\begin{align}
\CP_{\gamma,V_0}:=\LL_0\left(|x|+\frac{\S}{\gamma}\right)^\S+1.
\end{align}

\begin{theorem}\label{Thm:Exest}

Let $0<\V_1<\ldots<\V_N<1$. Then there is a constant $h=h(n,\LL_0,\S,\V_1)>0$ such that
\begin{align}
\left|\dot\CB V_0^N g\right|\leq h^N\left[\LL_0\left(|x|+\frac{\S}{\gamma}\right)^\S+1\right]^N\varPhi_{\r}|g|,
\end{align}
where $\gamma:=(1-\V_N)\V_1\cdots\V_{N-1}$ and $\r:=\V_1\cdots\V_N$.
\end{theorem}

\begin{proof}

Since $\gamma$ is fixed, let us write $\CP_{\gamma,V_0}=\CP_{\gamma}$. By Lemma \ref{Lem:Ex2}, we immediately obtain
\begin{align}\label{Est:BV1}
|\dot\CB V_0g|\leq c_0\CP_{1-\V_1}\varPhi_{\V_1}|g|.
\end{align}
Using this estimate then we get
\begin{align*}
\left|\dot\CB V_0^2g\right|&=\left|\dot\CB V_0(\dot\CB V_0g))\right|,\\
&\leq c_0^2\CP_{1-\V_1}\varPhi_{\V_1}(\CP_{1-\V_1}\varPhi_{\V_1}|g|).
\end{align*}
By Lemma \ref{Lem:tech} (i) and Lemma \ref{Lem:Ex1}, we have
\begin{align}\label{Est:P1V1}
\CP_{1-\V_1}&
\leq\LL_0|x|^\S+\LL_0\left(\frac{\S}{1-\V_1}\right)^\S+1\leq L\left(\LL_0|x|^\S+1\right),\quad L:=\LL_0\left(\frac{\S}{1-\V_1}\right)^\S+1.
\end{align}
Thus
\begin{align*}
\left|\dot\CB V_0^2g\right|&\leq c_0^2L\CP_{1-\V_1}\varPhi_{\V_1}\left((\LL_0|\cdot|^\S+1)\varPhi_{\V_1}|g|\right),\\
&=c_0^2L\CP_{1-\V_1}\int_{\R^n\times\R^n}(\LL_0|y|^\S+1)e^{-\V_1|x-y|}e^{-\V_1|y-z|}|g(z)|dzdy,\\
&\leq c_0^2L\CP_{1-\V_1}\int_{\R^n\times\R^n}\left[\LL_0\left(|x|+\frac{\S}{(1-\V_2)\V_1}\right)^\S+1\right]e^{-\V_1\V_2|x-y|}e^{-\V_1|y-z|}|g(z)|dzdy,\\
&\leq c_0^2L\CP_{1-\V_1}\CP_{(1-\V_2)\V_1}\varPhi_{\V_1\V_2}|g|\int_{\R^n}e^{-\V_1(1-\V_2)|y-z|}dy,\\
&\leq h^2\CP_{1-\V_1}\CP_{(1-\V_2)\V_1}\varPhi_{\V_1\V_2}|g|,
\end{align*}
where 
\[
h:=c_0L\max\left\{1,\int_{\R^n}e^{-\V_1(1-\V_1)|y|}dy\right\}
\]
and we have used Lemma \ref{Lem:Ex1} with $\V=\V_2$ and $\gamma=\V_1$ in the second inequality, and in the third inequality we applied the triangle inequality to get that
\[
-\V_1\V_2|x-y|\leq-\V_1\V_2|x-z|+\V_1\V_2|y-z|.
\]
Since $0<\V_1<\V_2<1$, it follows that $\CP_{1-\V_1}\leq\CP_{(1-\V_2)\V_1}$. So we obtain
\begin{align}\label{Est:BV2}
\left|\dot\CB V_0^2g\right|\leq\left(h\CP_{(1-\V_2)\V_1}\right)^2\varPhi_{\V_1\V_2}|g|,
\end{align}
which proves the desired estimate for the case $N=2$.\smallskip

Next, by using the estimates (\ref{Est:BV1}), (\ref{Est:P1V1}), and (\ref{Est:BV2}), we get
\begin{align*}
\left|\dot\CB V_0^3g\right|
&\leq c_0h^2\CP_{(1-\V_2)\V_1}^2\varPhi_{\V_1\V_2}\left(\CP_{1-\V_1}\varPhi_{\V_1}|g|\right),\\
&\leq c_0Lh^2\CP_{(1-\V_2)\V_1}^2\int_{\R^n\times\R^n}\left(\LL_0|y|^\S+1\right)e^{-\V_1\V_2|x-y|}e^{-\V_1|y-z|}|g(z)|dzdy,\\
&\leq c_0Lh^2\CP_{(1-\V_2)\V_1}^2\int_{\R^n\times\R^n}\left[\LL_0\left(|x|+\frac{\S}{(1-\V_3)\V_1\V_2}\right)^\S+1\right]e^{-\V_1\V_2\V_3|x-y|}e^{-\V_1|y-z|}|g(z)|dzdy,\\
&\leq c_0Lh^2\CP^2_{(1-\V_2)\V_1}\CP_{(1-\V_3)\V_1\V_2}\varPhi_{\V_1\V_2\V_3}|g|\int_{\R^n}e^{-\V_1(1-\V_2\V_3)|y-z|}dy,\\
&\leq h^3\CP_{(1-\V_3)\V_1\V_2}^3\varPhi_{\V_1\V_2\V_3}|g|.
\end{align*}
Therefore, we obtain by induction that
\begin{align*}
\left|\dot\CB V_0^Ng\right|\leq\left(h\CP_{(1-\V_N)\V_1\cdots\V_{N-1}}\right)^N\varPhi_{\V_1\cdots\V_N}|g|,
\end{align*}
which is the desired estimate.\QED

\end{proof}

\begin{remark}\
\begin{enumerate}
\item[(i)] The estimate in Theorem \ref{Thm:Exest} relies on the technical result Lemma \ref{Lem:tech} (i), Lemma \ref{Lem:Ex1}, and Lemma \ref{Lem:Ex2}, so it is true also if we replace $\CB$ with the operator $\CH$ satisfying (\ref{K}).
\item[(ii)] The result of Theorem \ref{Thm:Exest} means that $\CB V_0$ and $\CB V_0^N$ are essentially the multiplications by a power function to the usual convolution operator with kernel $e^{-\r|x|}$ ($0\leq\r<1$). Also observe that we can take $\r\to1$ arbitrarily close with the trade-off that $h\to\I$ and $\gamma\to0$.
\end{enumerate}

\end{remark}

For a time-dependent potential satisfying (\ref{Hyp:Vtime}), we have the following result.

\begin{corollary}\label{Cor:BVd}

Assume $V=V(x,t)$ satisfies (\ref{Hyp:Vtime}) and let $\t_0\in(0,T)$. Then for any $0<\V_1<\cdots<\V_N<1$, there is a constant $h=h(n,\|\LL\|_{L^\I(0,\t_0)},\S,\V_1)>0$ such that
\begin{align}
\left|\dot\CB V^Ng\right|\leq\left(h\CP_{\gamma,\O V_0}\right)^N\varPhi_{\r}|g|\quad\mbox{on $\R^n\times[0,\t_0]$},
\end{align}
where $\gamma=(1-\V_N)\V_1\cdots\V_{N-1}$, $\r=\V_1\cdots\V_N$, and
\begin{align}
\O V_0:=\|\LL\|_{L^\I(0,\t_0)}|x|^\S+1.
\end{align}
\end{corollary}

\begin{proof}

The proof follows from the preceding theorem together with the fact that
\[
\left|\dot\CB Vg\right|\leq\int_{\R^n}B(x-y)V(y,t)|g(y)|dy\leq\int_{\R^n}B(x-y)\O V_0(y,t)|g(y)|dy=\dot\CB\O V_0|g|
\]
and $|\dot\CB V^Ng|\leq\dot\CB\O V_0^N|g|$.\QED

\end{proof}

Now we establish the existence of one-parameter (semi)groups $e^{t\dot\CB V}$.

\begin{theorem}[One-parameter semigroup]\label{Thm:Green}

Let $V:\R^n\times[0,T)\to\R$ be a function satisfying (\ref{Hyp:Vtime}) and $\t_0\in(0,T)$. Assume that
\[
\S\in[0,1).
\]
Then for each $\r>0$ and any function $\vp\in L^1_{loc}(\R^n)$ such that $e^{-\r|\cdot|}\ast|\vp|<\I$ on $\R^n$, the series
\begin{align}
e^{t\dot\CB V}\vp:=\sum_{k=0}^\I\frac{t^k}{k!}\dot\CB V^k\vp
\end{align}
converges at each point on $\R^n\times[0,T)$. In fact, the series converges uniformly on compact subsets to a continuous function. 

\end{theorem}

\begin{proof}

We shall prove the convergence by selecting an increasing sequence of real numbers $0<\V_1<\V_2<\cdots<\V_k<\cdots<1$ and applying Corollary \ref{Cor:BVd}. The precise specifications about this sequence will be described during the proof. Let us denote
\begin{align*}
\gamma_k:=(1-\V_k)\V_1\cdots\V_{k-1},\quad\r_k:=\V_1\cdots\V_k.
\end{align*}
Fix $\t_0\in(0,T)$ and denote 
\[
\O V_0=\LL_0|x|^\S+1,\quad\LL_0=\|\LL\|_{L^\I(0,\t_0)},\quad\CP_{\gamma}=\CP_{\gamma,\O V_0}.
\]
By the triangle inequality and Corollary \ref{Cor:BVd} we can estimate the finite sum
\begin{align}\label{Sem:Est1}
\left|\sum_{k=0}^N\frac{t^k}{k!}\dot\CB V^k\vp\right|&\leq\sum_{k=0}^N\frac{t^k}{k!}\left|\dot\CB V^k\vp\right|\leq\sum_{k=0}^N\frac{t^k}{k!}\left(h\CP_{\gamma_k}\right)^k\varPhi_{\r_k}|\vp|=:\vp_N.
\end{align}
Also by Lemma \ref{Lem:tech} we have
\begin{align}\label{Sem:Est2}
\CP_{\gamma_k}^k=\left[\LL_0\left(|x|+\frac{\S}{\gamma_k}\right)^\S+1\right]^k
\leq(2\LL_0)^k|x|^{\S k}+2^k\left[\LL_0\left(\frac{\S}{\gamma_k}\right)^\S+1\right]^k.
\end{align}
From the estimates (\ref{Sem:Est1}) and (\ref{Sem:Est1}), we obtain
\begin{align*}
&|\vp_N|\leq J_N+K_N\quad\mbox{where}\,\,\begin{cases}
\displaystyle J_N:=\sum_{k=0}^N\frac{(2h\LL_0t)^k}{k!}|x|^{\S k}\varPhi_{\r_k}|\vp|,\\
\vspace{-10pt}\\
\displaystyle K_N:=\sum_{k=0}^N\frac{(2ht)^k}{k!}\left[\LL_0\left(\frac{\S}{\gamma_k}\right)^\S+1\right]^k\varPhi_{\r_k}|\vp|.
\end{cases}
\end{align*}

We now impose the first condition on $\V_k$. Assume $\V_k$ has the property that the infinite product 
\begin{align}
\r_\ast:=\prod_{k=1}^\I\V_k>0.
\end{align}
This can be achieved precisely when $\sum(1-\V_k)<\I$. Then $\varPhi_{\r_k}|\vp|\leq\varPhi_{\r_\ast}|\vp|$ for all $k$, so the following estimate for $J_N$ as $N\to\I$ is true:
\begin{align}\label{Linfty}
\limsup_{N\to\I}J_N&\leq L_\I\varPhi_{\r_\ast}|\vp|,\quad L_\I(x,t):=e^{2h\LL_0t|x|^\S}.
\end{align}
Next we consider $K_N$. If $\S=0$, then clearly
\[
K_N\to e^{2ht}\varPhi_{\r_\ast}|\vp|.
\]
Assume that $\S>0$. To obtain the convergence for $K_N$ we further restrict $\V_k$ as follows. Fix $\theta\in(0,1)$ small to be specified and choose $l>0$ so that $l-1/\S<-1$. The latter can be done because $\S\in(0,1)$. Now we take the sequence 
\begin{align}
\V_k=1-\theta k^{l-1/\S}\quad(k\geq1).
\end{align}
Since $\sum\theta\cdot k^{l-1/\S}<\I$, it follows that $\r_\ast>0$. In addition, we have $\gamma_k\geq(1-\V_k)\r_\ast=\theta k^{l-1/\S}\r_\ast$, hence the following estimate is true
\begin{align*}
\limsup_{N\to\I}K_N&\leq \sum_{k=0}^\I\frac{(2ht)^k}{k!}\left[\LL_0\left(\frac{\S}{\theta\r_\ast}\right)^\S k^{1-\S l}+1\right]^k\varPhi_{\r_\ast}|\vp|,\\
&\leq\sum_{k=0}^\I\frac{(2hFt)^k}{k!}k^{k(1-\S l)}\varPhi_{\r_\ast}|\vp|=:M_\I\varPhi_{\r_\ast}|\vp|,
\end{align*}
where
\begin{align}\label{Minfty}
M_\I(t):=\sum_{k=0}^\I\frac{(2hFt)^k}{k!}k^{k(1-\S l)},\quad F:=\LL_0\left(\frac{\S}{\theta\r_\ast}\right)^\S+1.
\end{align}
By the ratio test, then the series $M_\I$ is convergence for all $t$. \smallskip

Now we further analyse $\r_\ast$ so that $\varPhi_{\r_\ast}|u_0|<\I$. Note that $1-x\geq e^{-2x}$ for all $0\leq x\leq(\ln 2)/2$. By taking $\theta>0$ sufficiently small we have
\begin{align}
\r_\ast=\prod_{k=1}^\I(1-\theta k^{l-1/\S})\geq\exp\left(-2\theta\sum_{k=1}^\I k^{l-1/\S}\right).
\end{align}
Thus $\r_\ast$ can be arbitrary close to $1$ by taking $\theta\to0$. Moreover, we choose $0<\th\leq\th_0$ where
\begin{align}
\th_0=\frac{1}{2}\min\left\{\ln2,\frac{-\log\r}{\zeta\left(\frac{1}{\S}-l\right)}\right\},
\end{align}
where $\zeta$ is the Riemann zeta function, then we obtain $\r_\ast\geq\r$. So $\varPhi_{\r_\ast}|\vp|<\I$ on $\R^n$. This implies the pointwise convergence of $e^{t\dot\CB V}\vp=\lim_{N\to\I}\vp_N$ on $\R^n\times[0,\t_0]$. But $\t_0<T$ is arbitrary, therefore we have the convergence on $\R^n\times[0,T)$. The uniform convergence on compact subsets is obvious.\QED

\end{proof}

\begin{corollary}\label{Cor:EstBV}

Under the same conditions as Theorem \ref{Thm:Green}, the following estimate holds
\begin{align}
\left|e^{t\dot\CB V}\vp\right|\leq\varGamma(x,t;\S,\r,\|\LL\|_{L^\I(0,\t_0)})\,\varPhi_{\r}|\vp|\quad\mbox{on $\R^n\times[0,\t_0]$},
\end{align} 
where $\varGamma>0$ is the function defined by
\begin{align}
\varGamma=L_\I(x,t)+M_\I(t)
\end{align}
and $L_\I,M_\I$ are defined by (\ref{Linfty}) and (\ref{Minfty}) respectively.

\end{corollary}

\begin{remark}\

\begin{enumerate}
\item[(i)] The function $\varGamma$ behaves regularly if $\S$ and $\r$ are strictly less than 1. If the (minimal) parameter $\r\to1$ in the condition of $\vp$, we have $h\to\I$, $\th_0\to0$, and $L_\I,M_\I\to\I$; hence 
\[
\varGamma\to\I\quad\mbox{for each $(x,t)$}.
\]
The same conclusion holds if $\S\to1$ because $\th_0\to0$.
\item[(ii)] The proof of the preceding theorem fails when $\S=1$. In this case, it is not possible to select $l$ so that $l-1/\S<-1$ (which implies $\r_\ast>0$) and at the same time $1-\S l<1$ (which implies the convergence of $M_\I$).
\item[(iii)] Similar result as Theorem \ref{Thm:Green} holds if $\CB$ is replaced by $\CH$ satisfying (\ref{K}).
\end{enumerate}

\end{remark}

\section{Existence of solutions and comparison principles}\label{Sec:ExistComp}

In this section we prove some existence and uniqueness of solutions for the equation (\ref{Eqn:Lin}). Some comparison principles are also presented. We first show that the problem (\ref{Eqn:Lin}) with the potential
\begin{align}
a=\LL_0|x|^\S\,\,\mbox{where $\LL_0>0$, $\S\in[0,1)$ are constants},
\end{align}
admits a unique global solution (in the sense of (\ref{Eqn:Milds})) provided the initial function does not grow too fast. Here we do not put the positivity assumption on the initial function and the solution, but, as it will be revealed, the positivity is preserved. The result will be generalized to more general potentials later.\smallskip

Recall the notation from the previous section:
\begin{align*}
&V_0:=\LL_0|x|^\S+1,\quad 
\CP_{\gamma,V_0}:=\LL_0\left(|x|+\frac{\S}{\gamma}\right)^\S+1,\\
&\dot\CB V_0f=B\ast(V_0f),\quad\dot\CB V_0^N:=\dot\CB V_0^{N-1}\circ\dot\CB V_0.
\end{align*}
To find the solution $\mu:=e^tu$ of (\ref{Eqn:Milds}), we consider the Picard iteration scheme, i.e.\ consider
\begin{align*}
&U_1:=u_0+\int_0^t\dot\CB V_0u_0d\t=u_0+t\dot\CB V_0u_0,\\
&U_2:=u_0+\int_0^t\dot\CB V_0U_1d\t=u_0+t\dot\CB V_0u_0+\frac{t^2}{2!}\dot\CB V_0^2u_0,\\
&U_N:=u_0+\int_0^t\dot\CB V_0U_{N-1}d\t\quad(N\geq2),
\end{align*}
which by induction we get
\begin{align}\label{Series:UN}
U_N=\sum_{k=0}^N\frac{t^k}{k!}\dot\CB V_0^ku_0\quad(N\geq1).
\end{align}
It should be remarked that the above identity holds because of the time-independence of $V_0$ on $t$.

\begin{theorem}[Global well-posedness I]\label{Thm:ExConst}

Let $\LL_0>0$ and $\S\in[0,1)$. Assume $u_0\in L^1_{loc}(\R^n)$ and there is a constant $\r\in[0,1)$ such that 
\begin{align}
\int_{\R^n}e^{-\r|x-y|}|u_0(y)|dy<\I\quad\mbox{on $\R^n$}.
\end{align}
Then $\{U_N\}_{N=1}^\I$ in (\ref{Series:UN}) converges on $\R^n\times(0,\I)$ to a continuous function $\mu$. 
Moreover, $u=e^{-t}\mu$ is the unique global solution of (\ref{Eqn:Lin}) with $a(x,t)=\LL_0|x|^\S$, and, in addition, $u\geq0$ whenever $u_0\geq0$.

\end{theorem}

\begin{proof}
By Theorem \ref{Thm:Green} the first assertion is true and we can define
\begin{align}
\mu:=\lim_{N\to\I}U_N=e^{t\dot\CB V_0}u_0\quad\mbox{on $\R^n\times[0,\I)$}.
\end{align}
Also, the uniform convergence on compact subsets of $\{U_N\}_{N=1}^\I$ and the dominated convergence theorem give that the function $\mu$ satisfies
\begin{align}
\mu(t)=u_0+\int_0^t(\dot\CB V_0\mu)(\t)d\t.
\end{align}
Thus $u=e^{-t}\mu$ is a solution for (\ref{Eqn:Lin}) in the sense of (\ref{Eqn:Milds}). That $u_0$ implies $u\geq0$ is obvious. The uniqueness of solution is true by the comparison principle (Theorem \ref{Comp1}) below.\QED

\end{proof}


\begin{definition}[Time-independent potentials]

The \textit{Green operator} for (\ref{Eqn:Lin}) with $a=\LL_0|x|^\S$ where $\LL_0>0$ and $\S\in[0,1)$ is defined to be
\begin{align}
\CG_{a}(t)=e^{t\dot\CB(\Lap+\LL_0|\cdot|^\S)}:=e^{-t}e^{t\dot\CB(\LL_0|\cdot|^\S+1)}.
\end{align}
It is acting on functions belonging to the space
\[
\CE=\bigcup_{0\leq\r<1}\CE_\r\quad\mbox{where}\quad\CE_\r:=\left\{f\in L^1_{loc}(\R^n):\int_{\R^n}e^{-\r|x-y|}|f(y)|dy<\I\right\}=L^1(\R^n,e^{-\r|\cdot|}).
\]
\end{definition}

\begin{remark}\
\begin{itemize}
\item[(i)] By Theorem \ref{Thm:ExConst} and the preceding definition, if $a=\LL_0|x|^\S$ ($\LL_0>0,\S\in[0,1)$) and $u_0\in\CE$, then
\[
u=\CG_{a}(t)u_0
\]
is the unique global solution to in (\ref{Eqn:Lin}). Using the result of Corollary \ref{Cor:EstBV}, it follows that
\[
\left|\CG_{a}(t)u_0\right|\leq e^{-t}\varGamma(x,t;\S,\r,\LL_0)\,e^{-\r|\cdot|}\ast|u_0|.
\]
We remark that if $u_0\geq0$ then so is $u$, or in other words, $\CG_a(t)$ is a positive operator.
\item[(ii)] $\CE$ is a Fr\'echet space whose metric can be induced by the weighted $L^1$-norms.
\end{itemize}

\end{remark}

\begin{theorem}[Comparison principle I]\label{Comp1}

Let $\LL_0>0$, $\S\in[0,1)$. Assume $u_0,v_0\in\CE$ and $u,v\in C([0,T),\CE)$ satisfy
\begin{align}
\begin{cases}
\displaystyle u(t)\leq u_0+\int_0^t\CB\left((\Lap+\LL_0|\cdot|^\S)u(\t)\right)\,d\t,\\
\vspace{-10pt}\\
\displaystyle v(t)\geq v_0+\int_0^t\CB\left((\Lap+\LL_0|\cdot|^\S)v(\t)\right)\,d\t,
\end{cases}
\end{align}
for all $0<t<T$. We have
\begin{align}
u_0\leq v_0\quad\Rightarrow\quad u\leq v.
\end{align}

\end{theorem}

\begin{proof}
Let $V_0=\LL_0|x|^\S+1$, and $\mu=e^tu$, $\nu=e^tv$ which satisfy
\[
\mu(x,t)\leq u_0+\int_0^t(\dot\CB V_0\mu)(x,\t)\,d\t,\quad\nu(x,t)\geq v_0+\int_0^t(\dot\CB V_0\nu)(x,\t)\,d\t.
\]
We also define a function $w\in C([0,T),\CE)$ by
\begin{align}
w(x,t):=\max_{0\leq\t\leq t}(\mu-\nu)_+(x,t)\geq0.
\end{align}
Observe that if $s<\t$ then $w(x,s)\leq w(x,\t)$ for all $x\in\R^n$. This implies 
\begin{align}
(\dot\CB V_0^Nw)(x,s)\leq(\dot\CB V_0^Nw)(x,\t)\quad\mbox{for all $N\geq1$ and $x\in\R^n$}.
\end{align}
Consider
\begin{align*}
(\mu-\nu)(x,t)&\leq(u_0-v_0)+\int_0^t(\dot\CB V_0(\mu-\nu))(x,\t)d\t,\\
&\leq\int_0^t(\dot\CB V_0(\mu-\nu)_+)(x,\t)\,d\t\leq\int_0^t(\dot\CB V_0w)(x,\t)\,d\t.
\end{align*}
So we get a Gronwall type inequality:
\begin{align}\label{Gronwall}
w(x,t)\leq \int_0^t(\dot\CB V_0w)(x,\t)\,d\t.
\end{align}

\begin{claim}
For each integer $N\geq1$, we have
\begin{align}
w\leq\frac{\t^{N+1}}{(N+1)!}\dot\CB V_0^Nw\quad\mbox{on $\R^n\times(0,T)$}.
\end{align}

\end{claim}

\begin{proof}[\textbf{Claim}]
Since $(\dot\CB V_0w)(x,t)$ is increasing in $t$, we have by (\ref{Gronwall}) that
\begin{align}
w(x,t)&\leq(\dot\CB V_0w)(x,t)\int_0^td\t=\frac{t^2}{2!}(\dot\CB V_0w)(x,t),
\end{align}
which is the case $N=1$.\smallskip

Assume the claim is true for an integer $N=k\geq1$. Employing (\ref{Gronwall}), the induction hypothesis, and the fact that $(\dot\CB V_0^{k+1}w)(x,t)$ is increasing in $t$, then we get
\begin{align*}
w(x,t)&\leq\int_0^t\dot\CB V_0\left(\frac{\t^{k+1}}{(k+1)!}(\dot\CB V_0^kw)\right)(x,\t)\,d\t,\\
&\leq(\dot\CB V_0^{k+1}w)(x,t)\int_0^t\frac{\t^{k+1}}{(k+1)!}\,d\t=\frac{t^{k+2}}{(k+2)!}(\dot\CB V_0^{k+1}w)(x,t)
\end{align*}
hence the claim is true for all $N$.\QED

\end{proof}

We continue the proof of the comparison theorem. Fix time $t=\t_0>0$. Applying Theorem \ref{Thm:Green} to the function $\vp=w(\cdot,\t_0)\in\CE$ we find that the series $\sum_{k=0}^\I(\t_0^k/k!)\dot\CB V_0^k\vp$ is convergence. Thus
\begin{align}
\lim_{k\to\I}\frac{\t_0^k}{k!}(\dot\CB V_0^kw)(x,\t_0)=0,
\end{align}
which implies, using the claim above, that
\begin{align}
w(x,\t_0)=0.
\end{align}
This is true for all $x\in\R^n$ and $\t_0>0$, hence $w\equiv0$. We conclude that $u\leq v$.\QED

\end{proof}

Next, we generalize the results to time-dependent potentials. A motivation for the following result can be seen from the following simple  ODE:
\begin{align}
\frac{d}{dt}\vp=B(t)\vp,\quad\vp|_{t=0}=\vp_0\in\R.
\end{align}
By the variation of parameter (or the technique of finding integrating factor), we obtain the solution
\begin{align}
\vp(t)=e^{\int_0^tB(\t)\,d\t}\vp_0.
\end{align}

\begin{theorem}[Global well-posedness II]

Let $a\in L^1_{loc}(\R^n\times[0,\I))$ be a function satisfying (\ref{Hyp:Exist})
where $\LL\in B_{loc}([0,\I))$ and $\S\in[0,1)$. Then for each $u_0\in\CE$ the problem (\ref{Eqn:Lin}) admits a unique global solution $u\in C([0,\I);\CE)$. 
Furthermore, in the case of nonnegative potential, if $u_0\geq0$ then $u\geq0$.

\end{theorem}

\begin{proof}

Let us denote $V=a(x,t)+1$. We introduce the operators
\begin{align*}
\begin{cases}
\displaystyle
\widehat{\CB}V(t)g:=\int_0^t\dot{\CB}V(\t)g\,d\t=\int_0^t\int_{\R^n}B(x-y)V(y,\t)g(y)dy\,d\t,\\
\vspace{-10pt}\\
\displaystyle
\widehat{\CB}V(t)^Ng:=\widehat{\CB}V(t)^{N-1}(\widehat{\CB}V(t)g),\quad(N\geq2).
\end{cases}
\end{align*}
For example, when $N=2$ we have
\begin{align*}
\widehat{\CB}V(t)^2g&=\int_0^t\int_{\R^n}B(x-y)V(y,\t)(\widehat{\CB}V(\t)g)(y)dy\,d\t,\\
&=\int_0^t\int_{\R^n}\int_0^\t\int_{\R^n}B(x-y)B(y-z)V(y,\t)V(z,s)g(z)\,dzdsdyd\t.
\end{align*}
When $g$ also depends on $t$, we put $\widehat{\CB}V(t)g=\int_0^t\dot{\CB}V(\t)g(\t)\,d\t$. If $V$ is time-independent, then 
\begin{align}
\widehat{\CB}V^Ng=\frac{t^N}{N!}\dot{\CB}V^Ng.
\end{align}

Consider the iteration scheme
\begin{align*}
&U_1=u_0+\int_0^t\dot{\CB}V(\t)u_0\,d\t=u_0+\widehat{\CB}V(t)u_0,\\
&U_2=u_0+\int_0^t\dot{\CB}V(\t)U_1(\t)\,d\t,\\
&\hphantom{U_2(t)}=u_0+\widehat{\CB}V(t)U_1=u_0+\widehat{\CB}V(t)u_0+\widehat{\CB}V(t)^2u_0,
\end{align*}
and generally, for any $N\geq1$,
\begin{align}
U_N&=u_0+\int_0^t\dot{\CB}V(\t)U_{N-1}(\t)\,d\t=\sum_{k=0}^N\widehat{\CB}V(t)^ku_0.\label{Iteration1}
\end{align}

Fix $0<T<\I$. We prove the existence of solution for (\ref{Eqn:Lin}) on any $\R^n\times[0,T)$, where $0<T<\I$. For $(x,t)$ in this set, we have
\begin{align}
|V|\leq V_0:=\LL_0|x|^\S+1,\quad\LL_0=\max\{\LL(t):0\leq t\leq T\}.
\end{align}
All the calculations below will be taken at $(x,t)\in\R^n\times[0,T)$. We estimate
\begin{align*}
&|\widehat{\CB}Vg|\leq\dot{\CB}V_0|g|\int_0^td\t=t\dot{\CB}V_0|g|,\\
&|\widehat{\CB}V^Ng|\leq\frac{t^N}{N!}\dot{\CB}V_0^N|g|.
\end{align*}
Thus we have
\begin{align}\label{Est:Global2}
|U_N|&\leq\sum_{k=0}^N\frac{t^k}{k!}\dot{\CB}V_0^k|u_0|\to e^{t\dot{\CB}V_0}|u_0|\quad\mbox{as $N\to\I$}.
\end{align}
So we find that the sequence $\{U_N\}_{N=1}^\I$ converges to  a continuous function, and in fact it converges uniformly on every compact subset of $\R^n\times[0,T)$. The convergence and (\ref{Iteration1}) implies that the continuous function
\begin{align}
\mu:=\lim_{N\to\I}U_N
\end{align}
is a mild solution of (\ref{Eqn:Lin}) in the sense of (\ref{Eqn:Milds}). By the estimate (\ref{Est:Global2}) we have
\begin{align}
|\mu(x,t)|\leq e^{t\dot{\CB}V_0}|u_0|.
\end{align}
This implies in particular that $u=0$ whenever $u_0=0$. The uniqueness of solutions in the space $C([0,T);\CE)$ follows from the comparison principle below. In the case that $a(x,t)\geq0$, the nonnegativity of $u$ whenever $u_0\geq0$ is obvious.\QED

\end{proof}

\begin{definition}[Time-dependent potentials]

Assume (\ref{Hyp:Exist}). The \textit{Green operator} for the equation (\ref{Eqn:Lin}) is defined to be
\begin{align}
\CG_a(t)\vp=e^{\int_0^t\dot{\CB}(\Lap+a(\cdot,\t))\,d\t}\vp:=e^{-t}\sum_{k=0}^\I\widehat{\CB}(a(\cdot,t)+1)^k\vp\quad(\vp\in\CE).
\end{align}

\end{definition}

We also have the comparison principle for the case of positive time-dependent potentials.

\begin{theorem}[Comparison principle II] Let $a\in L^1_{loc}(\R^n\times[0,T))$ be such that
\begin{align}
0\leq a(x,t)\leq\LL(t)|x|^\S\quad\mbox{on $\R^n\times[0,T)$},
\end{align}
where $\LL\in B_{loc}([0,T))$ and $\S\in[0,1)$. Assume $u_0,v_0\in\CE$ and $u,v\in C([0,\I);\CE)$ satisfy
\begin{align}
\begin{cases}
\displaystyle
u(t)\leq u_0+\int_0^t\CB((\Lap+a(\cdot,\t))u(\t))\,d\t,\\
\vspace{-10pt}\\
\displaystyle
v(t)\geq v_0+\int_0^t\CB((\Lap+a(\cdot,\t))v(\t))\,d\t
\end{cases}
\end{align}
for all $t>0$. We have
\begin{align}
u_0\leq v_0\quad\Rightarrow\quad u\leq v.
\end{align}

\end{theorem}

\begin{proof}

The proof has the same idea as Theorem \ref{Comp1}. Let $V=a(x,t)+1$ and $\mu=e^tu,\nu=e^tv$ which satisfy
\begin{align*}
\mu\leq u_0+\int_0^t\dot{\CB}V(\t)\mu(\t)\,d\t,\quad\nu\geq v_0+\int_0^t\dot{\CB}V(\t)\nu(\t)\,d\t.
\end{align*}
Define the function $w$ as before:
\begin{align}
w(x,t)=\max_{0\leq\t\leq t}(\mu-\nu)_+\geq0.
\end{align}
Then $w(x,t)$ is increasing in $t$.
As in the proof of Theorem \ref{Comp1}, we get a Gronwall type inequality
\begin{align}
w(t)\leq\int_0^t\dot{\CB}V(\t)w(\t)\,d\t.
\end{align}
Consider the estimates:
\begin{align*}
&w(t)\leq\int_0^t\dot{\CB}V(\t)w(t)\,d\t=\widehat{\CB}V(t)w(t),\\
&w(t)\leq\int_0^t\dot{\CB}V(\t)(\widehat{\CB}V(\t)w(\t))\,d\t\leq\widehat{\CB}V(t)^2w(t),\\
&
\end{align*}
and generally we obtain
\begin{align}
w(t)\leq\widehat{\CB}V(t)^Nw(t)\quad\forall\,N\geq1.
\end{align}
By the convergence of series (\ref{Iteration1}) and (\ref{Est:Global2}) we can conclude that $w(x,t)=0$ for all $x\in\R^n,t>0$. Therefore we obtain $u\leq v$ as desired.\QED

\end{proof}

\section*{Appendix}

\addtocounter{section}{1}
\setcounter{equation}{0}

We consider the equation (\ref{Eqn:LConv}). Assume that $\mathbf{b}$ is $C^1$ in $x$. Then we rewrite this equation in the form
\begin{align}
\begin{cases}
\P_tu-\Lap\P_tu=\Lap u+\N\cdot(\mathbf{b}(x,t)u)+c(x,t)u&x\in\R^n,t>0,\\
\quad u(x,0)=u_0(x)&x\in\R^n,
\end{cases}
\end{align}
where
\begin{align}
c(x,t):=a(x,t)-\N\cdot\mathbf{b}(x,t).
\end{align}
Then we derive the non-local version of this equation by taking $(1-\Lap)^{-1}$, setting $\mu=e^tu$, then applying the integration by parts to get
\begin{align}\label{Eqn:muconv}
\P_t\mu=\dot{\CB}W\mu+\dot{\CD}\mathbf{b}\mu,
\end{align}
where $W=1+c(x,t)$ and $\dot{\CD}\mathbf{b}$ is the convolution multiplication operator given by
\begin{align}
\dot{\CD}\mathbf{b}\vp:=\int_{\R^n}\N_xB(x-y)\cdot\mathbf{b}(y,t)\vp(y)\,dy.
\end{align}

\begin{lemma}\label{Lem:conv}

The kernel function $\N B$ satisfies the the condition
\begin{align}
|\N B(x)|\in L^1_{loc}(\R^n),\quad|\N B(x)|\lesssim e^{-|x|}\quad\mbox{as $|x|\to\I$},
\end{align}
which is precisely (\ref{K}).
\end{lemma}

\begin{proof}
Let $r=|x|$. By the results in \cite{AronszajnSmith61}, we have
\begin{align*}
\frac{d}{dr}B(r)=c_{n}r^{1-\frac{n}{2}}K_{\frac{n}{2}}(r)
\end{align*}
and the asymptotic formulas
\begin{align*}
&K_{\frac{n}{2}}(r)\sim r^{-\frac{n}{2}}\quad\mbox{as $r\to0$},\\
&K_{\frac{n}{2}}(r)\sim r^{-\frac{1}{2}}e^{-r}\quad\mbox{as $r\to\I$}.
\end{align*}
Therefore we obtain $|B'|\in L^1_{loc}$ and $|B'|\lesssim e^{-r}$ as $r\to\I$ which implies the statement of this lemma.\QED
\end{proof}

\begin{theorem}

Assume that the potential and the coefficient $\mathbf{b}$ satisfy
\begin{align}
\max\{|a(x,t)|,|\mathbf{b}|(x,t),|\N\cdot\mathbf{b}|(x,t)\}\leq\LL(t)|x|^\S,
\end{align}
where $\LL\in B_{loc}([0,\I))$ and $\S\in[0,1)$. Then for each $u_0\in\CE$ the problem (\ref{Eqn:LConv}) admits a unique global solution.

\end{theorem}

\begin{proof}

We introduce the operator
\[
\CM(t)\vp(t)=\int_0^t\dot{\CB}W(\cdot,\t)\vp(\t)\,d\t+\int_0^t\dot{\CD}\mathbf{b}(\cdot,\t)\vp(\t)\,d\t
\]
and its iterations
\[
\CM(t)^N\vp(t):=\CM(t)^{N-1}(\CM(t)\vp(t))\quad(N\geq2).
\]
So the solution $\mu$ of (\ref{Eqn:muconv}) satisfies
\[
\mu(t)=u_0+\CM(t)\mu(t)\quad\mbox{for all $t>0$}.
\]
We introduce the iteration
\begin{align*}
&U_1(t)=u_0+\CM(t)u_0,\\
&U_2(t)=u_0+\CM(t)U_1(t),\\
&\hphantom{U_2(t)}=u_0+\CM(t)u_0+\CM(t)^2u_0,\\
&U_N(t)=u_0+\CM(t)U_{N-1}\quad(N\geq2),
\end{align*}
hence
\begin{align*}
U_N(t)=\sum_{k=0}^N\CM(t)^ku_0.
\end{align*}

The operator $\CM(t)$ has the following estimate
\begin{align*}
|\CM(t)\vp(t)|&\leq\int_0^t\dot{\CH}V(\cdot,\t)|\vp(\t)|\,d\t=\widehat{\CH}V(\cdot,t)|\vp|(t),
\end{align*}
where $\CH$ is the convolution operator with kernel function
\[
H(x):=B(x)+|\N B(x)|,
\]
and
\[
V:=W+|\mathbf{b}|.
\]
This implies
\begin{align*}
|\CM(t)^N\vp(t)|\leq\widehat{\CH}V(\cdot,t)^N|\vp|(t)\quad\mbox{for all $N\geq1$}.
\end{align*}
Applying Lemma \ref{Lem:conv} and the fact that $V\lesssim\LL(t)|x|^\S$, we can perform as in the proof of Theorem \ref{Thm:Exest} to conclude the pointwise convergence (uniformly on compact subsets) of 
\[
\mu:=\lim_{N\to\I}U_N.
\]
It then follows by the dominated convergence theorem that $\mu$ is a desired solution. For the uniqueness, one can perform a similar estimate to show that $\mu=0$ is the only solution in the case $u_0=0$.\QED

\end{proof}

\end{document}